\documentclass[12pt,a4paper,reqno]{amsart}
\usepackage{amssymb,amsfonts}
\usepackage{amsthm}
\usepackage{amsmath}
\usepackage{amscd}
\usepackage[latin2]{inputenc}
\usepackage{t1enc}
\usepackage[mathscr]{eucal}
\usepackage{indentfirst}
\usepackage{graphicx}
\usepackage{graphics}
\usepackage{pict2e}
\usepackage{epic}
\numberwithin{equation}{section}
\usepackage[margin=2.9cm]{geometry}
\usepackage{epstopdf} 

\usepackage{hyperref}

\usepackage{todonotes}

\theoremstyle{plain}

\numberwithin{equation}{section}
\newtheorem{theorem}{Theorem}[section]

\newtheorem{lemma}[equation]{Lemma}
\newtheorem{proposition}[equation]{Proposition}

\newtheorem{conjecture}[theorem]{Conjecture}

{
\theoremstyle{definition}

\newtheorem{remark}[theorem]{Remark}
}

\newcommand{\KL}{Kazhdan-Lusztig\ }
\DeclareMathOperator{\ch}{ch}
\newcommand{\D}{\mathcal{D}}
\renewcommand{\P}{\mathcal{P}}
\DeclareMathOperator{\rk}{rk}

\newcommand{\OS}{OS}

\newcommand{\grVRep}{\operatorname{grVRep}}

\newcommand{\VRep}{\operatorname{VRep}}
\newcommand{\OSc}{H}

\DeclareMathOperator{\Ind}{Ind}

\numberwithin{equation}{section}
 
\begin{document}

\title{Kazhdan-Lusztig polynomials of thagomizer matroids}

\author{Katie R. Gedeon}

\address{Department of Mathematics \\ University of Oregon \\
Eugene, OR 97403} 

\email{kgedeon@uoregon.edu}

%\acknowledgements{People}

%\subjclass[2010]{Primary: 05C??. Secondary: 05C??}

%\keywords{sample paper} 

\begin{abstract} We introduce thagomizer matroids and compute the Kazhdan-Lusztig polynomial of a rank $n+1$ thagomizer matroid by showing that the coefficient of $t^k$ is equal to the number of Dyck paths of semilength $n$ with $k$ long ascents. We also give a conjecture for the $S_n$-equivariant Kazhdan-Lusztig polynomial of a thagomizer matroid.
\end{abstract}

\maketitle

\section{Introduction} 
The main objects of study in this paper are the \KL polynomials of a particular family of matroids. The \KL polynomial of a matroid was introduced by Elias, Proudfoot and Wakefield \cite{EPW}. In the appendix of that paper, the authors (along with Young) explicitly computed the coefficients of these polynomials for some uniform and braid matroids of small rank. Proudfoot, Wakefield and Young studied uniform matroids of rank $n-1$ on $n$ elements \cite{PWY} and gave a combinatorial description for the coefficients of the associated \KL polynomial.

Let $M_n$ be the matroid associated with the graph obtained from the bipartite graph $K_{2,n}$ by adding an edge between the two distinguished vertices. We call $M_n$ a \textbf{thagomizer matroid}\footnote{The underlying graph is also called the complete tripartite graph $K_{1,1,n}$ or the fan graph $F_{n,2}$.}. The ground set of $M_n$ has size $2n+1$ and the rank of $M_n$ is $n+1$. We give a description of the flats of $M_n$ in Section \ref{sec:non-ekl}.

\begin{figure}[h]
    \includegraphics[scale=1]{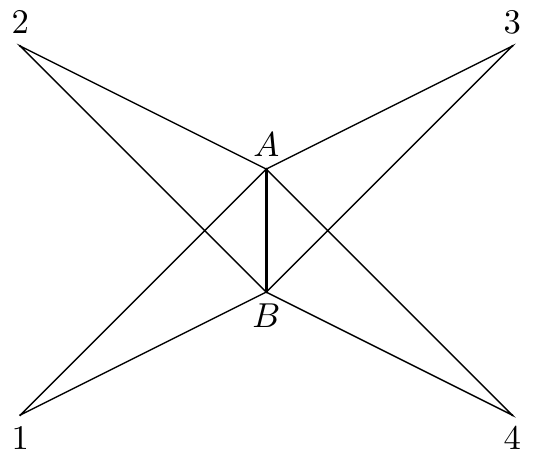}
    \caption{The underlying graph of $M_4$.}
    \label{fig:thag}
\end{figure}

Let $P_n(t)$ be the \textbf{\KL polynomial} of $M_n$ and set
\[
\Phi(t,u) := \sum_{n=0}^\infty P_{n}(t) u^{n+1}.
\]
Let $c_{n,k}$ be the $k$-th coefficient of $P_n(t)$ and note that the degree of $P_n(t)$ is $\lfloor \frac n2 \rfloor$. The following theorem is our main result.

\pagebreak
\begin{theorem}\label{coeffs} The following (equivalent) statements hold.
\begin{enumerate}
\item For all $n$ and $k$, $c_{n,k}$ is the number of Dyck paths of semilength $n$ with $k$ long ascents.
\item The generating function $\Phi(t,u)$ is equal to $ \dfrac{1-\sqrt{1-4u(1-u+tu)}}{2(1-u+tu)}$.
\end{enumerate}
\end{theorem}

\begin{remark}\label{rem:closed-form}
It is known that the number of Dyck paths of semilength $n$ with $k$ long ascents is equal to the quantity $\frac{1}{n+1} \binom{n+1}{k} \sum_{j=2k}^{n} \binom{j-k-1}{k-1} \binom{n+1-k}{n-j}$ (see \cite{Dyck} and sequence A091156 in \cite{oeis}). 
\end{remark}

\begin{remark}
The total number of Dyck paths of semilength $n$ is equal to the $n$-th Catalan number $\mathcal{C}_n = \frac{1}{n+1} \binom{2n}{n}$. Thus Theorem \ref{coeffs} implies that $P_n(1) = \mathcal{C}_n$ and Remark \ref{rem:closed-form} implies that the leading coefficient of $P_{2n}(t)$ is $\mathcal{C}_n$. Interestingly, $\mathcal{C}_n$ also appears as the leading coefficient of the \KL polynomial of the uniform matroid of rank $2n-1$ on $2n$ elements (see \cite{EPW} Appendix A and \cite{PWY}). 
\end{remark}

\begin{remark}
Prior to this paper, uniform matroids were the only infinite family of matroids for which the \KL polynomial has been computed. For example, it is still an open problem to compute the \KL polynomial of the braid matroid; see \cite{EPW} and \cite{ekl} for partial results.
\end{remark}

%Motivated by geometry, i.e. intersection cohomology, Proudfoot, Wakefield and Young also gave a categorified description of the coefficients of the \KL polynomial of a uniform matroid of rank $n-1$ on $n$ elements \cite{PWY} given by an irreducible representation of $S_n$. This lead to a definition of the equivariant \KL polynomial by the author, Proudfoot and Young \cite{ekl} where we further studied uniform matroids in this context and compute the $S_n$-equivariant \KL polynomials of braid matroids of small rank. The $S_n$-equivariant \KL polynomial of 

%We conclude this section with a description of the structure of the paper. Section \ref{background} reviews the definitions of the non-equivariant and equivariant \KL polynomials as well as the definition of Dyck paths. In Section \ref{sec:non-ekl} we prove Theorems \ref{main} and \ref{coeffs}. We conclude by exploring the $S_n$-equivariant \KL polynomial of $M_n$ and giving a conjecture for its coefficients in Section \ref{sec:ekl}.

We conclude this section with a description of the structure of the paper. In Section \ref{background}, we recall the definition of the \KL polynomial of a matroid and review Dyck paths. Section \ref{sec:non-ekl} is dedicated to proving Theorem \ref{coeffs}.

In Section \ref{sec:ekl}, we recall the definition of the equivariant \KL polynomial of a matroid and explore the $S_n$ action on $M_n$ which allows us to make a conjecture for the $S_n$-equivariant \KL polynomial of $M_n$. This categorification of \KL coefficients was first considered for a uniform matroid of rank $n-1$ on $n$ elements by Proudfoot, Wakefield and Young \cite{PWY} where they were given by an irreducible representation of $S_n$. The equivariant \KL polynomial for a general matroid was subsequently defined by the author, Proudfoot and Young \cite{ekl} where we further studied uniform matroids in this context and computed the $S_n$-equivariant \KL polynomials of braid matroids of small rank.

\subsection*{Acknowledgements}
The author would like to thank Nicholas Proudfoot and Benjamin Young for many helpful discussions and suggestions. The success of this project is due in large part to the On-Line Encyclopedia of Integer Sequences \cite{oeis} and SageMath \cite{sagemath}. The author was partially supported by NSF grant DMS-1565036.

\section{Preliminaries}\label{background}

\subsection{\KL polynomials of a matroid}\label{sec:kl-background}
In this section we follow \cite{EPW} to define the non-equivariant \KL polynomial of a matroid (which we simply refer to as the \KL polynomial).

Let $M$ be a matroid on the finite ground set $E$. Denote by $L(M)$ the \textbf{lattice of flats} of $M$ and $\chi_M(t)$ the \textbf{characteristic polynomial} of $M$. The matroid $M^F$ is called the \textbf{contraction} of $M$ at $F$; it is the matroid on the ground set $E\setminus F$ whose lattice of flats is  $L^F:=\{ G\setminus F \mid G\in L(M) \text{ and } G\geq F\}$. The matroid $M_F$ is called the \textbf{localization} of $M$ at $F$ and is the matroid with ground set $F$ whose lattice of flats is $L_F := \{ G\in L(M) \mid G \leq F\}$. 
\bigskip

The \textbf{\KL polynomial} $P_M(t) \in \mathbb{Z}[t]$ is characterized by the following three properties \cite[Theorem 2.2]{EPW}. 
\begin{itemize}
\item If $\rk M=0$, then $P_M(t)=1$.\\
\item If $\rk M>0$, then $\deg P_M(t) < \frac 12 \rk M$.\\
\item For every $M$, $t^{\rk M}P_M(t^{-1}) = \displaystyle\sum_F \chi_{M_F}(t) P_{M^F}(t)$.
\end{itemize}

%\begin{theorem}[\cite{EPW}, Theorem 2.2]
%There is a unique way to assign to each matroid $M$ a polynomial $P_M(t) \in \mathbb{Z}[t]$ such that the following conditions are satisfied:\\
%\begin{enumerate}\item If $\rk M =0, P_M(t) =1$.\\\item If $\rk M>0, \deg P_M(t) < \frac 12 \rk M$.\\\item For every $M, t^{\rk M}P_M(t^{-1}) = \displaystyle\sum_F \chi_{M_F}(t) P_{M^F}(t)$.\end{enumerate}
%\end{theorem}

%\begin{theorem}[\cite{ekl}, Theorem 2.8]
%There is a unique way to assign to each equivariant matroid $W \curvearrowright M$ an element $\P^W_M(t)\in\grVRep(W)$, such that the following conditions are satisfied:\begin{enumerate}\item If $\rk M = 0$,  $\P^W_M(t)$ is equal to the trivial representation in degree 0.\\\item If $\rk M > 0$,  $\deg \P^W_M(t) < \frac{1}{2}\rk M$.\\\item For every $M$, $\displaystyle t^{\rk M} \P^W_M(t^{-1}) = \sum_{[F] \in L/W}\Ind_{W_F}^W\left(\OSc^{W_F}_{M_F}(t) \otimes \P^{W_F}_{M^F}(t)\right).$\\\item Given a homomorphism $\varphi:W'\to W$, $\P^{W'}_M(t) = \varphi^* \P^W_M(t)$.\end{enumerate}
%\end{theorem}

\subsection{Dyck paths}\label{Dyck-background}
A \textbf{Dyck path} of semilength $n$ is a lattice path in $\mathbb{N}^2$ beginning at $(0,0)$ and ending at $(2n,0)$ with up-steps of the form $u=(1,1)$ and down-steps of the form $d=(1,-1)$. Such a Dyck path may be expressed as a word $\alpha \in \{u,d\}^{2n}$.

\begin{figure}[h]
    \includegraphics[scale=1]{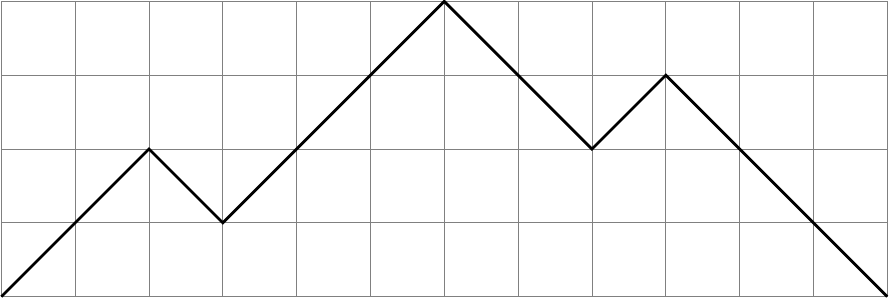}
    \caption{The Dyck path $uuduuudduddd$.}
    \label{fig:path}
\end{figure}%

\noindent
A \textbf{long ascent} of a Dyck path is an ascent of length at least 2. Equivalently, a long ascent of a Dyck path $\alpha$ is a maximal subword consisting of at least two consecutive $u$'s. The Dyck path given in Figure \ref{fig:path} has two long ascents.

Let $\D_n$ be the set of all Dyck paths of semilength $n$. We denote by $a_{n,k}$ the number of elements in $\D_n$ with exactly $k$ long ascents. As noted in \cite{Dyck}, $a_{n,k}$ is also the number of words $\alpha \in \D_n$ with $k$ occurrences of the subword $uud$. Additional interpretations of $a_{n,k}$ are known; see sequence A091156 in \cite{oeis}.
%\nicktodo{Mention that $a_{n,k}$ is the number of paths of semilength $n$ with $k$ long ascents, descent number, and also the number of occurrences of $\tau$}

\section{Main results}\label{sec:non-ekl}

We begin this section with a description of the flats $F \in L(M_n)$ given by the underlying graph. Let $AB$ be the distinguished edge. For any $j \in \{1,\ldots, n\}$, we call the subgraph with edges $Aj$ and $ Bj$ a \textbf{spike}. 

If $\rk F = i$, then either
\begin{enumerate}
\item $F$ contains exactly one edge from $i$ distinct spikes, or
\item $F$ is the union of $i-1$ spikes and $AB$.
\end{enumerate}
For example, when $n=4$, a rank $2$ flat of the first type is given by $\{A1,B3\}$ and a rank $2$ flat of the second type is given by $\{AB,A4,B4\}$ (see Figure \ref{fig:thag}).

%In each case, we consider only the structure of the lattice of flats of the matroid obtained from $M_n$ by localization and contraction of the flat. 
In the first case, the localization $(M_n)_F$ yields a Boolean matroid of rank $i$, and the contraction $M_n^F$ gives a matroid whose lattice of flats is isomorphic to that of $M_{n-i}$. In the second case, the localization $(M_n)_F$ gives a matroid whose lattice of flats is isomorphic to that of $M_{i-1}$, and the contraction $M_n^F$ is a Boolean matroid of rank $n-i+1$.
%In the first case, the localization $(M_n)_F$ is a Boolean matroid, and the contraction $M_n^F$ is isomorphic to $M_{n-i}$. In the second case, the localization $(M_n)_F$ is isomorphic to $M_{i-1}$, and the contraction $M_n^F$ is a Boolean matroid.

%In the first case, the localization $(M_n)_F$ is a Boolean matroid of rank $i$, and the contraction $M_n^F$ has lattice of flats isomorphic to that of $M_{n-i}$. In the second case, localization $(M_n)_F$ has lattice of flats isomorphic to that of $M_{i-1}$, and the contraction $M_n^F$ is a Boolean matroid of rank $n-i+1$.

The characteristic polynomial of a rank $i$ Boolean matroid is equal to $(t-1)^i$. For thagomizer matroids, it is clear that $\chi_{M_i}(t) = (t-1)(t-2)^i$ by a simple deletion/contraction argument.

%If $F$ is of the first type and $\rk F=n-i$, there are $\binom n{n-i}$ ways to choose the spikes and $2^{n-i}$ choices of edges. If $F$ is of the second type and $\rk F =i$, there are only $\binom ni$ choices. 
Recall that we've set 
\[
P_n(t) := P_{M_n}(t) \quad \text{and} \quad \Phi(t,u) := \sum_{n=0}^\infty P_{n}(t) u^{n+1}.
\]
We first turn our attention towards proving the following lemma.

\begin{lemma}\label{main}
We have the following (equivalent) equations.
\begin{enumerate}
\item For all $n$, $t^{n+1}P_n(t^{-1}) = (t-1)^{n+1} + \displaystyle\sum_{i=0}^n \binom{n}{i} 2^{n-i}(t-1)^{n-i}P_i(t)$.
\\
\item $\Phi(t^{-1},tu) = \dfrac{ut-u}{1+u-tu} + \Phi\left(t, \dfrac{u}{1+2u-2tu}\right)$.
\end{enumerate}
\end{lemma}
\begin{proof}
There are $\binom n{i} \cdot 2^{n-i}$ flats of the first type of rank $n-i$ and $\binom ni$ flats of the second type of rank $i+1$. Note that for any Boolean matroid $M,$ $P_M(t)=1$ \cite[Corollary 2.10]{EPW}. Then we have
\begin{align} \label{eqn:kl}
t^{n+1} P_n(t^{-1})& = \sum_{i=0}^n \binom{n}{i} \Big( 2^{n-i}(t-1)^{n-i} P_i(t) + (t-1)(t-2)^i \Big) \\
&=(t-1)^{n+1} + \displaystyle\sum_{i=0}^n \binom{n}{i} 2^{n-i}(t-1)^{n-i}P_i(t) \nonumber
\end{align}
which is the formula given in Lemma \ref{main}(1). Now our defining recursion tells us that
\begin{align*}
\Phi(t^{-1},tu) &= \sum_{n= 0}^\infty P_n(t^{-1}) t^{n+1} u^{n+1}\\
 &= \sum_{n= 0}^\infty (t-1)^{n+1}u^{n+1} + \sum_{n= 0}^\infty\sum_{i=0}^n \binom{n}{i} 2^{n-i}(t-1)^{n-i}P_i(t)u^{n+1}.
\end{align*}
We let $m=n-i$ which allows us to write the second summand as
\[
%\Phi(t^{-1},tu) = \sum_{n= 0}^\infty (t-1)^{n+1}u^{n+1} + \sum_{i=0}^\infty P_i(t) u^{i+1} \sum_{m=0}^\infty 2^m \binom{m+i}{i} (t-1)^m u^m.
\sum_{i=0}^\infty P_i(t) u^{i+1} \sum_{m=0}^\infty 2^m \binom{m+i}{i} (t-1)^m u^m.
\]
Recall the identity
\[
\sum_{\ell=0}^\infty \binom{r+\ell}{r}x^\ell = \frac{1}{(1-x)^{r+1}} 
\]
and set $\ell =m$ and $x = 2u(t-1)$. This gives
\begin{align*}
\Phi(t^{-1},tu) &= u(t-1)\sum_{n= 0}^\infty (t-1)^{n}u^{n} + \sum_{i=0}^\infty  \frac{P_i(t) u^{i+1}}{(1-2u(t-1))^{i+1}}\\
&= \frac{u(t-1)}{1-u(t-1)} + \sum_{i=0}^\infty P_i(t)  \left( \frac{u}{1-2u(t-1)} \right)^{i+1}\\
&=  \dfrac{ut-u}{1+u-tu} + \Phi\left(t, \dfrac{u}{1+2u-2tu}\right).
\end{align*}
This completes the proof of Lemma \ref{main}. 
\end{proof}

Finally we are ready to prove Theorem \ref{coeffs}. Let $a_{n,k}$ be as in Section \ref{Dyck-background}, and set
\[
F(t,u) := \sum_{n,k\geq 0} a_{n,k} t^k u^n.
\]
It was shown in \cite[Section 1]{Dyck} that $F(t,u)$ satisfies
\[
u(1-u+tu)\cdot (F(t,u))^2 - F(t,u) + 1 = 0
\]
which gives
\[
F(t,u) = \frac{1 -\sqrt{1-4u(1-u+tu)}}{2u(1-u+tu)}.
\]
\textit{A priori}, this formula should have a $\pm$ sign. However, a plus sign would not give $F(t,u)$ as a formal power series. Hence we use a negative sign instead.
%However, if there were a plus sign, the coefficient of $tu^3$ would be $-8$ and clearly $a_{3,1}$ must be positive. Hence $F(t,u)$ should involve a negative sign instead. 

Let $f(t,u) := u\cdot F(t,u)$. Since we'd like to show that $\Phi(t,u) = u\cdot F(t,u)$, we first check that $f(t,u)$ satisfies the functional equation in Lemma \ref{main}(2).

We have
\[
f(t,u) = \frac{1-\sqrt{1-4u(1-u+tu)}}{2(1-u+tu)}
\]
and hence
\begin{align*}
f(t^{-1},tu) &= \frac{1-\sqrt{1-4tu(1-tu+u)}}{2(1-tu+u)}\\
&=\frac{ut-u}{1-tu+u} + \frac{1-2ut+2u-\sqrt{1-4tu(1-tu+u)}}{2(1-tu+u)}\\
&= \frac{ut-u}{1-tu+u} + \frac{1-\frac{1}{1+2u-2tu}\sqrt{1-4tu(1-tu+u)}}{\frac{2(1+u-tu)}{1+2u-2tu}}\\
&=\frac{ut-u}{1-tu+u} + \frac{1-\sqrt{1-\frac{4u\left( 1+2u-2tu-u+tu \right) }{(1+2u-2tu)^2} }}{\frac{2(1+u-tu)}{1+2u-2tu}}\\
&=\frac{ut-u}{1-tu+u}  + f\left( t, \frac{u}{1+2u-2tu} \right).
\end{align*}

Lastly, we note that both $c_{n,k}$ and $a_{n,k}$ are zero if $n > 2k$ and that $f(t,0) = \Phi(t,0) =1$. Then $f(t,u) = \Phi(t,u)$ which equivalently tells us that $c_{n,k} = a_{n,k}$. This completes the proof of Theorem \ref{coeffs}.\hfill \qed 
%\bigskip

\section{The $S_n$ action}\label{sec:ekl}
Recall the notation set in Section \ref{sec:kl-background}. That is, let $M$ be a matroid on a finite ground set $E$. Given a flat $F\in L(M)$, let $M^F$ denote the contraction of $M$ at $F$ and let $M_F$ denote the localization of $M$ at $F$.

Let $W$ be a finite group acting on $E$ and preserving $M$.  We refer to the data $\{ M,E,W \}$ as an \textbf{equivariant matroid} $W \curvearrowright M$. For any $F,G\in L(M)$, let $W_F\subseteq W$ be the stabilizer of $F$ and let $W_{FG} := W_F\cap W_G$. Note that the action of $W$ on $M$ induces an action of $W_F$ on both $M_F$ and $M^F$. Let $\VRep(W)$ be the ring of isomorphism classes of virtual representations of $W$ and set
\[
\text{grVRep}(W) := \text{VRep}(W)\otimes_{\mathbb{Z}}\mathbb{Z}[t]. %\quad \text{and} \quad \text{grRep}(W) := \text{Rep}(W)\otimes_{\mathbb{N}}\mathbb{N}[t].
\]
Let $OS^W_{M,i} \in \text{Rep}(W)$ be the degree $i$ part of the Orlik-Solomon algebra of $M$. The \textbf{equivariant characteristic polynomial} of $M$, $H_M^W(t)$, is given by
\[
H^W_M(t) := \sum_{p=0}^{\rk M} (-1)^p t^{\rk M - p} \OS^W_{M,p} \in \grVRep(W).
\]
Note that the equivariant characteristic polynomial $H_M^W(t)$ is a categorified version of the usual characteristic polynomial $\chi_M(t)$. That is, we can recover $\chi_M(t)$ from $H_M^W(t)$ by taking the graded dimension.

The \textbf{equivariant \KL polynomial} of $W \curvearrowright M$, denoted $\P_M^W(t)$, is a categorified version of the \KL polynomial and is characterized by the following three properties \cite[Theorem 2.8]{ekl}.
\begin{itemize}
\item If $\rk M = 0$,  $\P^W_M(t)$ is equal to the trivial representation in degree 0.\\
\item If $\rk M > 0$,  $\deg \P^W_M(t) < \frac{1}{2}\rk M$.\\
\item For every $M$, $\displaystyle t^{\rk M} \P^W_M(t^{-1}) = \sum_{[F] \in L/W}\Ind_{W_F}^W\left(\OSc^{W_F}_{M_F}(t) \otimes \P^{W_F}_{M^F}(t)\right).$\\
%\item Given a homomorphism $\varphi:W'\to W$, $\P^{W'}_M(t) = \varphi^* \P^W_M(t)$.
\end{itemize}
The polynomial $\P_M^W(t)$ is an element of $\grVRep(W)$ and we can recover $P_M(t)$ from $\P_M^W(t)$ by taking the graded dimension.

Now we turn our attention back to the thagomizer matroid $M_n$. Though the full automorphism group of $M_n$ is $S_n\times S_2$ (unless $n=1$ in which case it is $S_3$), here we only consider the action of the symmetric group $S_n$. Let 
\[
\P_n(t) := \P^{S_n}_{M_n}(t) \quad \text{and} \quad \phi(t,u) := \sum_{u=0}^\infty \P_n(t)u^{n+1}.
\]
Let $\Upsilon_n$ be all partitions of $n$ of the form $[a, n-a-2i-{\eta}, 2^i, \eta]$ where $\eta \in \{0,1\}$, $i\geq 0$ and $1<a < n$. For any partition $\lambda$ of $n$, we let $V_\lambda$ be the irreducible representation of $S_n$ indexed by $\lambda$. %In this section we will explore the results that allow us to make the following conjecture.

For any partition $\lambda$, we set \[
\kappa(\lambda) = \left\{
\begin{array}{l l}
\lambda_1 - \lambda_2 +1 \quad & \lambda \neq [n-1,1]\\
\lambda_1-1  & \text{otherwise}
\end{array}
\right.
\] 
and
\[
\omega(\lambda) = \left\{
\begin{array}{l l}
1 \quad & \lambda_{\ell(\lambda)} \neq 1 \\
0 & \text{otherwise.}
\end{array}
\right.
\]

\begin{conjecture}\label{equiv} For all $n>0$, we have
%We have $C_{n,0} = V_{[n]}$ and for all $k>0$,
%\begin{enumerate}
%\item 
$$\P_n(t) = \sum_{\lambda \in \Upsilon_n} \kappa(\lambda) V_\lambda t^{\ell(\lambda)-1} (t+1)^{\omega(\lambda)} + V_{[n]} ((n-1)t+1).$$
%\item $p_n(t) = \sum_{k>0} C_{n,k}t^k$ where
%\[C_{n,k} = \sum_{\substack{\lambda \in \Upsilon_n \\ \ell(\lambda) = i+1}} \! \! \! \kappa(\lambda) V_{\lambda}\]
%\end{enumerate}

\end{conjecture}
\begin{remark} 
We have checked this conjecture for thagomizer matroids of rank at most 20 using SageMath \cite{sagemath}. For our calculations, we worked in the symmetric function setting (see Proposition \ref{ekl-rec}).
\end{remark}

\begin{remark}
We know the coefficients of $\P_n(t)$ will be honest representations by \cite[Corollary 2.12]{ekl}  since $M_n$ is $S_n$-equivariantly realizable.
\end{remark}

%\begin{remark}
%Applying the hook-length formula to compute the dimension of $V_\lambda$ in Conjecture \ref{equiv} produces a polynomial with non-negative integer coefficients. One can check, using a computer, that the $k$-th coefficient of this polynomial is the same as $c_{n,k}$ of Theorem \ref{coeffs}(1). We have also checked Conjecture \ref{equiv} in this way for thagomizer matroids of rank at most 20.
%\end{remark}

\begin{remark}
Unlike the analogous statements for uniform matroids, Conjecture \ref{equiv} is less enlightening than Theorem \ref{coeffs}(1) (see \cite{ekl}, Theorem 3.1 and Remark 3.4). That is, the coefficients of the \KL polynomial of a uniform matroid are more cleanly expressed when given as the dimension of a certain representation of the symmetric group. This is not the case for thagomizer matroids.
\end{remark}

The remainder of this section is devoted to understanding the results that allow us to derive the recursive formula and functional equation for the Frobenius characteristic of $\mathcal{P}_n(t)$. Let
\[
W(t) := (t-1)\mathbb{C} \quad \text{and} \quad V(t):=(t-2)\mathbb{C}
\]
as virtual vector spaces. Then $W(t)^{\otimes r}$ is the equivariant characteristic polynomial of a rank $r$ Boolean matroid and $W(t)\otimes V(t)^{\otimes r}$ is the equivariant characteristic polynomial of $M_r$. Both $W(t)^{\otimes r}$ and $W(t)\otimes V(t)^{\otimes r}$ are virtual representations of $S_r$, where $S_r$ acts by permuting the factors of the graded tensor product. Note that the equivariant \KL polynomial of a Boolean matroid is the trivial representation in degree zero.

We'd like to categorify the recursive formula given in Lemma \ref{main}(1). Recall Equation \ref{eqn:kl}
\[t^{n+1} P_n(t^{-1}) =\sum_{i=0}^n \binom{n}{i}2^{n-i} (t-1)^{n-i} P_i(t) + \sum_{i=0}^n \binom{n}{i}(t-1)(t-2)^i.\]
The first sum is over flats of rank $n-i$ of the first type mentioned in Section \ref{sec:non-ekl}. For flats of this type, summing over $[F]\in L(M_n)/S_{n}$ gives
\setcounter{equation}{4}
\begin{equation}\label{eqn:ekl1}
\sum_{m+j+i=n} \Ind_{S_m \times S_j \times S_i}^{S_n} \left( W(t)^{\otimes m} \otimes W(t)^{\otimes j} \otimes \P_i(t)\right) \quad \in \grVRep(S_n)
\end{equation}
where $S_m$ permutes the vertices that are connected to $A$ by an edge in $F$, $S_j$ permutes the vertices that are connected to $B$ by an edge in $F$, and $S_i$ permutes the vertices that are not adjacent to any edge in $F$.
Similarly, summing over flats of the second type gives
\begin{equation}\label{eqn:ekl2}
\sum_{i=0}^n \Ind_{S_{i}\times S_{n-i}}^{S_n} \left( W(t)\otimes V(t)^{\otimes i} \right)\quad \in \grVRep(S_n)
\end{equation}
where $S_{n-i}$ is acting trivially. 

As in \cite[Section 3.1]{ekl}, we now translate to symmetric functions. We consider the \textbf{Frobenius characteristic}
\[
\ch : \grVRep(S_n) \overset{\sim}{\longrightarrow} \Lambda_n[t]
\]
where $\Lambda_n$ is the space of symmetric functions of degree $n$ in infinitely many formal variables $\{x_i\mid i \in \mathbb{N}\}$.

Let $s[\lambda] := \ch V_\lambda$ be the Schur function corresponding to $\lambda$ and set
\[
p_n(t) :=\ch \P_n(t), \quad w_n(t):= \ch W(t)^{\otimes n} \quad \text{and} \quad v_n(t) := \ch V(t)^{\otimes n}.
\]
Applying Frobenius characteristic to Equations \ref{eqn:ekl1} and \ref{eqn:ekl2}, we obtain 
\[
t^{n+1}p_n(t^{-1}) = (t-1)\displaystyle\sum_{\ell =0}^n v_\ell(t) s[n-\ell] + \sum_{i+j+m=n} p_i(t) w_j(t) w_m(t).
\]

Finally, we pass to generating functions, working in the ring $\Lambda[[t,u]]$ of formal power series in the variables $\{t,u,x_1, x_2,  \ldots\}$ that are symmetric in the $x$ variables.We let
\[
s(u)  := \sum_n s[n]u^n ,\quad w(t,u) := \sum_n w_n(t)u^n
\]
and
\[
v(t,u) := \sum_n v_n(t)u^n.
\]
Note that 
\[
w(t,u) = \frac{s(tu)}{s(u)}
\]
by \cite[Proposition 3.9]{ekl}.
The results of this section can be summarized in the following proposition.
\begin{proposition}\label{ekl-rec}
We have the following (equivalent) equations.
\begin{enumerate}
\item For $n >0,$ $t^{n+1}p_n(t^{-1}) = (t-1)\displaystyle\sum_{\ell =0}^n v_\ell(t) s[n-\ell] + \sum_{i+j+m=n} p_i(t) w_j(t) w_m(t)$.
\\
\item $\phi(t^{-1},tu) = (t-1)us(u)v(t,u) + w(t,u)^2 \phi(t,u)$. 
\end{enumerate}
\end{proposition}

\setcounter{theorem}{7}
\begin{remark}
In \cite{ekl}, we were able to compute the equivariant \KL polynomial for uniform matroids by showing that our ``guess'' satisfied a recursion analogous to the one found in Proposition \ref{ekl-rec}(2). That case was much simpler; we only had to consider singular applications of the Pieri rule. In this case, $w(t,u)^2$ requires multiple applications of the Pieri rule while $v_n(t) = s[n] \big[ v_1(t) \big]$ involves a plethysm. This makes proving Conjecture \ref{equiv} much more difficult.
\end{remark}

%\begin{remark}
%One reason we have only considered the $S_n$ action on $M_n$ here is because we were unable to produce recursive formulae like those of Proposition \ref{ekl-rec} for the $S_n\times S_2$-equivariant \KL polynomial.
%\end{remark}

\bibliographystyle{amsalpha}

\bibliography{thag-bib}

\end{document}